\documentclass[11]{amsart}

\usepackage{amsfonts}
\usepackage{amscd}
\usepackage{amsmath, mathrsfs, amssymb}
\usepackage{amsthm}
\usepackage{setspace}
\usepackage{hyperref}
\usepackage{color}
\usepackage{epsfig}
\usepackage{here}
\usepackage{graphicx}
\usepackage[all]{xy}
\usepackage{psfrag}
\usepackage{graphicx,transparent}
\usepackage{enumerate}
\usepackage{caption}

\theoremstyle{plain}
\newtheorem{theorem}{Theorem}[section]

\newtheorem{proposition}[theorem]{Proposition}
\newtheorem{corollary}[theorem]{Corollary}

\theoremstyle{definition}
\newtheorem{remark}[theorem]{Remark}

\newcommand{\MM}{\mathcal M}

\newcommand{\BM}{\overline{\mathcal M}}

\newcommand{\PP}{\mathcal P}

\newcommand{\OO}{\mathcal O}

\newcommand{\BPP}{\overline{\mathcal P}}

\newcommand{\even}{\operatorname{even}}
\newcommand{\odd}{\operatorname{odd}}

\newcommand{\nonhyp}{\operatorname{nonhyp}}

\newcommand{\irr}{\operatorname{irr}}

\newcommand{\reg}{\operatorname{reg}}

\newcommand{\bbC}{\mathbb C}

\newcommand{\Lin}{\operatorname{Lin}}
\newcommand{\Nfold}{\operatorname{Nfold}}

\newcommand{\bbQ}{\mathbb Q}

\newcommand{\bbZ}{\mathbb Z}

\begin{document}

\title[]{Nonvarying, affine, and extremal geometry of strata of differentials}

\date{\today}

\author{Dawei Chen}

 \thanks{Research partially supported by National Science Foundation Grant DMS-2001040
        and Simons Foundation Collaboration Grant 635235.}
\address{Department of Mathematics, Boston College, Chestnut Hill, MA 02467, USA}
\email{dawei.chen@bc.edu}

\begin{abstract}
We prove that the nonvarying strata of abelian and quadratic differentials in \cite{ChenMoeller1, ChenMoeller2} have trivial tautological rings and are affine varieties. We also prove that strata of $k$-differentials of infinite area are affine varieties for all $k$. Vanishing of homology in degree higher than the complex dimension follows as a consequence for these affine strata. Moreover we prove that the stratification of the Hodge bundle for abelian and quadratic differentials of finite area is extremal in the sense that merging two zeros in each stratum leads to an extremal effective divisor in the boundary.  A common feature throughout these results is a relation of divisor classes in strata of differentials as well as its incarnation in Teichm\"uller dynamics.  
\end{abstract}

\maketitle

\setcounter{tocdepth}{1}
\tableofcontents

%%%%%%%%%%%%%%%%%
\section{Introduction}
\label{sec:intro}
%%%%%%%%%%%%%%%%%

For $\mu = (m_1, \ldots, m_n)$ with $\sum_{i=1}^n m_i  = k( 2g-2)$, let $\PP_g^k(\mu)$ be the projectivized stratum of $k$-differentials $\omega$ on smooth and connected complex curves of genus $g$ that have labeled zeros and poles whose orders are specified by $\mu$. Equivalently, $\PP_g^k(\mu)$ parameterizes $k$-canonical divisors of prescribed type $\mu$. 

The study of differentials is important in surface dynamics and moduli theory. We refer to \cite{Zorich, Wright, ChenSurvey} for an introduction to this subject. 
Despite various known results, the global geometry of $\PP_g^k(\mu)$ remains quite mysterious, e.g., the full structure of the tautological ring of $\PP_g^k(\mu)$ is largely unknown. Moreover, it is generally unclear whether $\PP_g^k(\mu)$ is an affine variety. The birational geometry of $\PP_g^k(\mu)$ is also less studied. 
Such questions are meaningful for understanding a moduli space, e.g., many exciting ideas and results have been discovered in the study of these questions for the moduli space $\MM_g$ of curves of genus $g$.  

There has been some expectation among experts that the (rational) Picard group of $\PP_g^k(\mu)$ should behave similarly comparing to that of $\MM_g$, which is of rank one (for $g\geq 3$) generated by the tautological divisor class $\lambda$. However, our result below shows that such an expectation should at least exclude the strata whose Teichm\"uller curves have nonvarying sums of Lyapunov exponents (as listed in \cite{ChenMoeller1, ChenMoeller2} for about $30$ of them).  For simplicity we call them nonvarying strata. Note that a stratum can be disconnected due to spin and hyperelliptic structures (\cite{KontsevichZorich, Lanneau}). Since different connected components can have distinct properties, when speaking of a stratum we mean a specific connected component according to the hyperelliptic and spin parity labelings. We say that the tautological ring of a stratum is trivial if it is isomorphic to $\bbQ$ generated by the fundamental class, i.e., any tautological class of positive codimension is zero over $\bbQ$.   

\begin{theorem}
\label{thm:nonvarying}
The nonvarying strata of abelian and quadratic differentials listed in \cite{ChenMoeller1, ChenMoeller2} have trivial tautological rings and are affine varieties.  
\end{theorem}

In contrast if a stratum is varying, then the tautological classes should not be all trivial (see Remarks~\ref{remark:varying} and~\ref{remark:speculation}). 

Next we consider $k$-differentials of infinite area, i.e., when $\mu$ contains at least one entry $\leq -k$ (see \cite{BCGGM-k} for an introduction to $k$-differentials).  

\begin{theorem}
\label{thm:infinite}
Strata of $k$-differentials of infinite area are affine varieties.  
\end{theorem}

Affine varieties are Stein spaces in analytic geometry.  By \cite{AndreottiFrankel, Narasimhan} we thus obtain the following result about the homology of an affine stratum of differentials. 

\begin{corollary}
\label{cor:homology}
Let $\PP_g^k(\mu)$ be one of the above affine strata of differentials. Then $H_d(\PP_g^k(\mu), \bbZ) = 0$ if $d > \dim_{\bbC} \PP_g^k(\mu)$ and $H_d(\PP_g^k(\mu), \bbZ)$ is torsion free if $d = \dim_{\bbC} \PP_g^k(\mu)$. 
\end{corollary}

\begin{remark}
For strata of holomorphic differentials with $\mu = (m_1, \ldots, m_n)$ and $m_i \geq 0$ for all $i$, a (strongly) $(g+1)$-convex exhaustion function was constructed on $\PP_g^1(\mu)$ by using the area and length functions on flat surfaces (see~\cite[Proposition 3.17]{Mondello}), i.e., $\PP_g^1(\mu)$ is a $(g+1)$-complete complex space. It follows that $H_d(\PP_g^1(\mu), \bbZ)$ is torsion and $H_d(\PP_g^1(\mu), \bbC) = 0$ for $d> \dim_{\bbC} \PP_g^1(\mu) + g = 3g-2+n$ (see~\cite[Theorem 0.3]{BallicoBolondi}).\footnote{Note that the $q$-complete index is defined in \cite{Mondello} one higher than in~\cite{BallicoBolondi}. } 
\end{remark}

In general little is known about the homology of strata of differentials (nevertheless see~\cite{CMZEuler} for computing the Euler characteristics and~\cite{Zykoski} for the $L^{\infty}$-isodelaunay decomposition). Even if one is interested in strata of holomorphic differentials in high genus, differentials in low genus and meromorphic differentials naturally appear in the boundary of compactified strata (\cite{BCGGM1, BCGGM3}). Therefore, we expect that the above results can be useful for inductively computing the homology of strata of differentials.   

Finally we turn to the stratification of the (ordinary) Hodge bundle of holomorphic differentials (up to scale) as the disjoint union of $\PP_g^1\{\mu\}$, where $\mu$ runs over all partitions of $2g-2$ and $\{ \mu \}$ denotes the version of strata with unlabeled zeros. Note that $\PP_g^1\{ \mu\}$ is a finite quotient of $\PP_g^1(\mu)$ under the group action permuting the zeros that have the same order.  We also denote by $\BPP_g^1\{\mu\}$ the closure of $\PP_g^1\{\mu\}$ in the extended Hodge bundle over the Deligne--Mumford moduli space $\BM_g$ of stable curves. We use similar notations for the stratification of the (quadratic) Hodge bundle of quadratic differentials with at worst simple poles and the extension over $\BM_{g,s}$, where $s$ is a given number of simple poles to start with.  

\begin{theorem}
\label{thm:extremal}
Let $\mu = (m_1, \ldots, m_n)$ be a signature of abelian differentials with $m_i \geq 1$ for all $i$, 
and let $\mu' = (m_1 + m_2, m_3, \ldots, m_n)$. Then the stratum $\BPP_g^1\{\mu'\}$ is an extremal effective divisor in $\BPP_g^1\{\mu\}$. 

Let $\mu = (m_1, \ldots, m_n)$ be a signature of quadratic differentials with $m_i \geq -1$ for all $i$, 
 and let $\mu' = (m_1 + m_2, m_3, \ldots, m_n)$, where at least one of $m_1, m_2$ is positive. Then the stratum $\BPP_g^2\{\mu'\}$ is an extremal effective divisor in $\BPP_g^2\{\mu\}$. 
\end{theorem}

Differentials of type $\mu'$ as degeneration of differentials of type $\mu$ can be realized geometrically by merging the two zeros of order $m_1$ and $m_2$. Theorem~\ref{thm:extremal} thus has an amusing interpretation that the stratification of the Hodge bundle (for abelian and quadratic differentials of finite area) is extremal, where merging two zeros leads to an extremal effective divisor in the boundary. We remark that if a stratum of type $\mu$ is disconnected, then Theorem~\ref{thm:extremal} holds for each connected component of type $\mu$ and the component of type $\mu'$ obtained by merging the zeros in the prescribed component of type $\mu$. 
 
\subsection*{Strategies of the proofs} 
To prove Theorem~\ref{thm:nonvarying} we show that the nonvarying property yields an extra relation between tautological divisor classes which forces them to be trivial. To prove Theorem~\ref{thm:infinite} we observe the sign change of $m_i+k$ for entries $m_i\leq -k$ and $m_i > -k$ in $\mu$, and use it to exhibit an ample divisor class which is trivial on strata of $k$-differentials of infinite area.  To prove Theorem~\ref{thm:extremal} we show that Teichm\"uller curves contained in 
$\PP_g^k\{\mu'\}$ have negative intersection numbers with the divisor class of $\BPP_g^k\{\mu'\}$ in $\BPP_g^k\{\mu\}$ for $k = 1, 2$.

\subsection*{Related works} 
Previously in \cite{ChenAffine} the author showed that the nonvarying strata of abelian differentials for $\mu = (4)^{\odd}$, $(3,1)$, $(2,2)^{\odd}$, and $(6)^{\even}$ are affine by using the geometry of canonical curves in low genus (see also~\cite{LooijengaMondello} for a related discussion). This is now completed for all known nonvarying strata (including for quadratic differentials). Moreover in \cite{ChenAffine} the author also showed that there is no complete curve in any stratum of $k$-differentials of infinite area, which now follows from Theorem~\ref{thm:infinite}. Note that in general the implication does not work the other way around, e.g., $\mathbb A^2$ removing a point contains no complete curves but it is not affine. For unprojectivized strata of holomorphic differentials, nonexistence of complete curves was shown in \cite{Gendron} and \cite{ChenComplete}. Finally the beginning case of Theorem~\ref{thm:extremal} for the principal strata, i.e., for all zeros equal to one, was previously established in \cite{Gheorghita}. Now this extremal behavior is shown to hold for all zero types. Therefore, our current results can further enhance the understanding of the tautological ring, affine geometry, topology, and birational geometry of the strata. We hope that the ideas employed in this paper can shed some light on related questions.  

\subsection*{Acknowledgements} The author would like to thank Matteo Costantini, Iulia Gheorghita, Samuel Grushevsky, Richard Hain, Martin M\"oller, Gabriele Mondello, and Scott Mullane for helpful discussions on relevant topics. 

%%%%%%%%%%%%%%%%%
\section{Triviality criteria}
\label{sec:criteria}
%%%%%%%%%%%%%%%%%

Marking the zeros and poles of differentials, $\PP_g^k(\mu)$ can be viewed as a subvariety in the moduli space $\MM_{g,n}$ of curves of genus $g$ with $n$ marked points. We introduce the following tautological divisor classes that will be used throughout the paper.  Let $\eta$ be the first Chern class of the tautological line bundle $\OO(-1)$ on $\PP_g^k(\mu)$ whose fibers are spanned by the underlying $k$-differentials $\omega$. Let $\lambda$ be the first Chern class of the Hodge bundle, $\kappa$ be the Miller--Morita--Mumford class, and $\psi_i$ be the cotangent line bundle class associated to the $i$-th marked point. When $m_i \neq -k$ for all $i$, we will often use the following quantity 
$$\kappa_\mu = \frac{1}{k}(2g-2+n)  -\sum_{i=1}^n \frac{1}{m_i+k}.$$ 

\begin{proposition}
\label{prop:disjoint}
Suppose $m_i \neq -k $ for all the entries of $\mu$. Let $D$ be an effective divisor in $\MM_{g,n}$ 
with divisor class $a\lambda + \sum_{i=1}^n b_i \psi_i$. Then the pullback of $D$ to $\PP_g^k(\mu)$ has divisor class 
$$ \frac{1}{12}\Big(\frac{(2g-2+n)a}{k} + \sum_{i=1}^n \frac{12b_i  -a  }{m_i +k}\Big) \eta. $$
Moreover if the above coefficient of $\eta$ is nonzero and if $D$ and $\PP_g^k(\mu)$ are disjoint, then $\eta$ is trivial on $\PP_g^k(\mu)$ and the tautological ring of $\PP_g^k(\mu)$ is trivial.  
\end{proposition}

For the nonvarying strata of abelian and quadratic differentials, the desired divisors $D$ in Proposition~\ref{prop:disjoint} arise from those used in \cite{ChenMoeller1, ChenMoeller2} which are disjoint from the respective strata. Occasionally we will use some variants of $\MM_{g,n}$ by marking fewer points or lifting to the spin moduli spaces, for which the above reasoning still works as we only consider the interior of these moduli spaces.   

\begin{proof}
The following relations hold on $\PP_g^k(\mu)$ (see \cite[Proposition 2.1]{ChenTauto}): 
$$\eta = (m_i + k) \psi_i, \quad 12\lambda = \kappa = \kappa_\mu \eta. $$
Then the claim on the pullback divisor class follows from these relations.  

For the other claim, we only need to show that in this case $\eta$ being trivial implies that the tautological ring of $\PP_g^k(\mu)$ is trivial, which indeed follows from the fact that $\eta$ generates the tautological ring when $\mu$ contains no entries equal to $-k$ (see \cite[Theorem 1.1]{ChenTauto}).  
\end{proof}

\begin{proposition}
\label{prop:affine}
If all tautological divisor classes are trivial on $\PP_g^k(\mu)$, then $\PP_g^k(\mu)$ is an affine variety. 
\end{proposition}

\begin{proof}
The tautological divisor class $\kappa + \sum_{i=1}^n \psi_i$ is ample on $\BM_{g,n}$ (see \cite[Theorem (2.2)]{Cornalba}). Hence $\kappa + \sum_{i=1}^n \psi_i$ being trivial on $\PP_g^k(\mu)$ implies that this ample divisor class can be supported on the boundary of $\PP_g^k(\mu)$. This proves the claim. 
\end{proof}

%%%%%%%%%%%%%%%%%
%%%%%%%%%%%%%%%%%
\section{Nonvarying strata}
\label{sec:nonvarying}
%%%%%%%%%%%%%%%%%
%%%%%%%%%%%%%%%%%

In this section we prove Theorem~\ref{thm:nonvarying}.  We use the effective divisors described in \cite{ChenMoeller1, ChenMoeller2} and then apply Propositions~\ref{prop:disjoint} and~\ref{prop:affine}. 

%%%%%%%%%%%%%%%%%
\subsection{The hyperelliptic strata}
%%%%%%%%%%%%%%%%%

Note that $\lambda$ is trivial on the hyperelliptic locus (see \cite{CornalbaHarris}), which implies that the hyperelliptic strata have trivial tautological rings and are affine varieties.  

%%%%%%%%%%%%%%%%%
\subsection{$\PP_3^1(4)^{\odd}$ and $\PP_3^1(3,1)$}
%%%%%%%%%%%%%%%%%

The divisor $H$ parameterizing hyperelliptic curves in $\MM_3$ is disjoint from these strata and has divisor class 
 $9\lambda$.   

%%%%%%%%%%%%%%%%%
\subsection{$\PP_3^1(2,2)^{\odd}$}
%%%%%%%%%%%%%%%%%

For $(X, p_1, p_2)\in \PP_3^1(2,2)^{\odd}$, by definition $p_1 + p_2$ is an odd theta characteristic, which maps $\PP_3^1(2,2)^{\odd}$ to the odd spin moduli space $\mathcal S_3^-$. The divisor $Z_3$ in $\mathcal S_3^-$ parameterizes odd theta characteristics with a double zero. It 
is disjoint from $\PP_3^1(2,2)^{\odd}$ and has divisor class $11\lambda$.   

%%%%%%%%%%%%%%%%%
\subsection{$\PP_3^1(2,1,1)$}
%%%%%%%%%%%%%%%%%

Marking the first two zeros maps $\PP_3^1(2,1,1)$ to $\MM_{3,2}$. The Brill--Noether divisor $BN_{3,(1,2)}^1$ in $\MM_{3,2}$ parameterizes curves that admit a $g^1_3$ given by $p_1 + 2p_2$. It is disjoint from $\PP_3^1(2,1,1)$ and has 
divisor class $-\lambda + \psi_1 + 3\psi_2$.    

%%%%%%%%%%%%%%%%%
\subsection{$\PP_4^1(6)^{\even}$}
%%%%%%%%%%%%%%%%%

For $(X, p)\in \PP_4^1(6)^{\even}$, by definition $3p$ is an even theta characteristic.  The theta-null divisor $\Theta_4$ in $\MM_{4,1}$ parameterizes curves that admit an odd theta characteristic whose support contains the marked point.  
It is disjoint from $\PP_4^1(6)^{\even}$ and has divisor class $30\lambda + 60 \psi_1$.     

%%%%%%%%%%%%%%%%%
\subsection{$\PP_4^1(6)^{\odd}$ and $\PP_4^1(5,1)$}
%%%%%%%%%%%%%%%%%

Marking the zero of the largest order, these strata map to $\MM_{4,1}$. The Brill--Noether divisor $BN_{3,(2)}^1$ in $\MM_{4,1}$ parameterizes curves that admit a $g^1_3$ ramified at the marked point. It is disjoint from both strata and has divisor class $8\lambda + 4\psi_1$.   

%%%%%%%%%%%%%%%%%
\subsection{$\PP_4^1(3,3)^{\nonhyp}$}
%%%%%%%%%%%%%%%%%

The divisor $\Lin_3^1$ in $\MM_{3,2}$ parameterizes curves that admit a $g^1_3$ with a fiber containing both marked points. It is disjoint from 
$\PP_4^1(3,3)^{\nonhyp}$ and has divisor class $8\lambda - \psi_1 - \psi_2$.   

%%%%%%%%%%%%%%%%%
\subsection{ $\PP_4^1(2,2,2)^{\odd}$}
%%%%%%%%%%%%%%%%%

For $(X, p_1, p_2, p_3)\in \PP_4^1(2,2,2)^{\odd}$, by definition $p_1 + p_2 + p_3$ is an odd theta characteristic.  
The divisor $Z_4$ in the odd spin moduli space $\mathcal S_4^-$ parameterizes odd theta characteristics with a double zero.  It 
is disjoint from $\PP_4^1(2,2,2)^{\odd}$ and has divisor class $12\lambda$.   

%%%%%%%%%%%%%%%%%
\subsection{ $\PP_4^1(3,2,1)$}
%%%%%%%%%%%%%%%%%

The Brill--Noether divisor $BN_{4,(1,1,2)}^1$ in $\MM_{4,3}$ parameterizes curves that admit a $g^1_4$ given by $p_1 + p_2 + 2p_3$. It is disjoint from 
$\PP_4^1(3,2,1)$ and has divisor class $-\lambda + \psi_1 + \psi_2 + 3\psi_3$.   

%%%%%%%%%%%%%%%%%
\subsection{ $\PP_5^1(8)^{\even}$}
%%%%%%%%%%%%%%%%%

The Brill--Noether divisor $BN_{3}^1$ in $\MM_{5,1}$ parameterizes curves with a $g^1_3$. It is disjoint from $\PP_5^1(8)^{\even}$ and has divisor class 
 $8\lambda$.   

%%%%%%%%%%%%%%%%%
\subsection{ $\PP_5^1(8)^{\odd}$}
%%%%%%%%%%%%%%%%%

The divisor $\Nfold_{5,4}^1(1)$ in $\MM_{5,1}$ parameterizes curves that admit a $g^1_4$ with a fiber containing $3p$. It is disjoint from 
$\PP_5^1(8)^{\odd}$ and has divisor class $ 7\lambda + 15\psi_1$.   

%%%%%%%%%%%%%%%%%
\subsection{$\PP_5^1(5,3)$}
%%%%%%%%%%%%%%%%%

The Brill--Noether divisor $BN_{4,(1,2)}^1$ in $\MM_{5,2}$ parameterizes curves that admit a $g^1_4$ with a fiber containing $p_1 + 2p_2$. It is disjoint from 
$\PP_5^1(5,3)$ and has divisor class $7\lambda + 7\psi_1 + 2\psi_2$.       

%%%%%%%%%%%%%%%%%
\subsection{$\PP_4^1(4,2)^{\even}$, $\PP_4^1(4,2)^{\odd}$, and $\PP_5^1(6,2)^{\odd}$}
\label{subsec:HN}
%%%%%%%%%%%%%%%%%

The nonvarying property of these strata was proved in \cite{YuZuo} by using the Harder--Narasimhan filtration of the Hodge bundle (see also \cite[Appendix A]{ChenMoeller2}).  Note that the filtration holds not only on a Teichm\"uller curve but also on (the interior of) each stratum.  Let $L_\mu$ be the sum of (nonnegative) Lyapunov exponents which satisfies that $\lambda = L_\mu \eta$ from the Harder--Narasimhan filtration on these strata.  In these cases we know that $L_\mu \neq \kappa_\mu / 12$.  Then the triviality of $\eta$ on these strata follows from the two distinct relations $\lambda = (\kappa_\mu/12) \eta$ and $\lambda = L_\mu \eta$.  

%%%%%%%%%%%%%%%%%
\subsection{Nonvarying strata of quadratic differentials in genus one and two}
%%%%%%%%%%%%%%%%%

Strata in genus $\leq 2$ can be generally dealt with as in Sections~\ref{subsec:g1} and~\ref{subsec:g2} below. 

%%%%%%%%%%%%%%%%%
\subsection{ $\PP^2_3(9,-1)^{\irr}$}
%%%%%%%%%%%%%%%%%

This stratum is disjoint from the hyperelliptic divisor $H$ in $\MM_3$ whose divisor class is $9\lambda$. 

%%%%%%%%%%%%%%%%%
\subsection{ $\PP^2_3(8)$, $\PP^2_3(7,1)$, $\PP^2_3(8,1,-1)$, $\PP^2_3(10, -1, -1)^{\nonhyp}$, and $\PP^2_3(9,-1)^{\reg}$}
%%%%%%%%%%%%%%%%%

Marking the zero of the largest order, these strata map to $\MM_{3,1}$. The divisor $W$ of Weierstrass points in $\MM_{3,1}$ 
is disjoint from these strata and has divisor class $ -\lambda + 6\psi_1$. 

%%%%%%%%%%%%%%%%%
\subsection{ $\PP^2_3(6,2)^{\nonhyp}$, $\PP^2_3(6,1,1)^{\nonhyp}$, $\PP^2_3(5,3)$, $\PP^2_3(5,2,1)$, 
$\PP^2_3(4,4)$, $\PP^2_3(4,3,1)$, $\PP^2_3(5,4,-1)$, $\PP^2_3(5,3,1,-1)$, $\PP^2_3(7,2,-1)$, 
$\PP^2_3(7,3, -1, -1)$, and $\PP^2_3(6,3,-1)^{\reg}$}
%%%%%%%%%%%%%%%%%

Marking the first two zeros, these strata map to $\MM_{3,2}$. The Brill--Noether divisor $BN^1_{3,(2,1)}$ in $\MM_{3,2}$ parameterizes curves that admit a $g^1_3$ given by $2p_1 + p_2$. It is disjoint from these strata and has divisor class $-\lambda + 3\psi_1 + \psi_2$. 

%%%%%%%%%%%%%%%%%
\subsection{$\PP^2_3(4,2,2)$, $\PP^2_3(3,3,2)^{\nonhyp}$, $\PP^2_3(4,3,2,-1)$, $\PP^2_3(3,2,2,1)$, $\PP^2_3(3,3,1,1)^{\nonhyp}$, and $\PP^2_3(3,3,3,-1)^{\reg}$}
%%%%%%%%%%%%%%%%%

Marking the first three zeros, these strata map to $\MM_{3,3}$. The Brill--Noether divisor $BN^1_{3,(1,1,1)}$ in $\MM_{3,3}$ parameterizes curves that admit a $g^1_3$ given by $p_1 + p_2 + p_3$. It is 
disjoint from these strata and has divisor class $-\lambda + \psi_1 + \psi_2 + \psi_3$. 

%%%%%%%%%%%%%%%%%
\subsection{ $\PP^2_4(13, -1)$, $\PP^2_4(11, 1)$, and $\PP^2_4(12)^{\reg}$}
%%%%%%%%%%%%%%%%%

Marking the zero of the largest order, these strata map to $\MM_{4,1}$. The divisor $W$ of Weierstrass points in $\MM_{4,1}$
is disjoint from these strata and has divisor class $-\lambda + 10\psi_1$. 

%%%%%%%%%%%%%%%%%
\subsection{ $\PP^2_4(10, 2)^{\nonhyp}$, $\PP^2_4(8,4)$, $\PP^2_4(8,3,1)$, and $\PP^2_4(9,3)^{\reg}$}
%%%%%%%%%%%%%%%%%

Marking the first two zeros, these strata map to $\MM_{4,2}$. The Brill--Noether divisor $BN^1_{4,(3,1)}$ in $\MM_{4,2}$ parameterizes curves that admit a $g^1_4$ given by $3p_1 + p_2$. It is disjoint from these strata and has divisor class $-\lambda + 6\psi_1 + \psi_2 $.  

%%%%%%%%%%%%%%%%%
\subsection{$\PP^2_4(7,5)$ and $\PP^2_4(6,6)^{\reg}$}
%%%%%%%%%%%%%%%%%

The Brill--Noether divisor $BN^1_{4,(2,2)}$ in $\MM_{4,2}$ parameterizes curves that admit a $g^1_4$ given by $2p_1 + 2p_2$. It is disjoint from these strata and has divisor class $-\lambda + 3\psi_1 + 3\psi_2$. 

%%%%%%%%%%%%%%%%%
\subsection{$\PP^2_4(7,3,2)$, $\PP^2_4(5,4,3)$, and $\PP^2_4(6,3,3)^{\reg}$}
%%%%%%%%%%%%%%%%%

The Brill--Noether divisor $BN^1_{4,(2,1,1)}$ in $\MM_{4,3}$ parameterizes curves that admit a $g^1_4$ given by $2p_1 + p_2 + p_3$. It is disjoint from 
these strata and has divisor class $-\lambda + 3\psi_1 + \psi_2 + \psi_3 $. 

%%%%%%%%%%%%%%%%%
\subsection{ $\PP^2_4(3,3,3,3)^{\reg}$}
%%%%%%%%%%%%%%%%%

It is disjoint from the Brill--Noether divisor $BN^1_{4,(1,1,1,1)}$ in $\MM_{4,4}$ whose divisor class is $-\lambda + \psi_1 + \psi_2 + \psi_3 + \psi_4$.  

%%%%%%%%%%%%%%%%%
\subsection{ $\PP^2_3(6,3,-1)^{\irr}$, $\PP^2_4(12)^{\irr}$, and $\PP^2_4(9,3)^{\irr}$}
%%%%%%%%%%%%%%%%%

We can use the Harder--Narasimhan filtration of the quadratic Hodge bundle for these strata (see \cite[Appendix A]{ChenMoeller2}) and argue as in~Section~\ref{subsec:HN}. 

%%%%%%%%%%%%%%%%%
\begin{remark}
\label{remark:varying}
%%%%%%%%%%%%%%%%%

We say that a stratum of abelian differentials $\PP^1_g(\mu)$ is varying, if it contains two Teichm\"uller curves that have distinct sums of Lyapunov exponents. In this case we claim that $\eta$ is nontrivial on $\PP^1_g(\mu)$. 

To see this, let $\Delta_0$ be the locus of stable differentials of type $\mu$ on curves with separating nodes only, where the differentials have simple poles with opposite residues at the nodes. If $\PP^1_g(\mu)$ is irreducible, then $\Delta_0$ is also irreducible and modeled on $\PP^1_{g-1}(\mu, \{-1,-1\})$.  Moreover, every Teichm\"uller curve in $\PP^1_g(\mu)$ only intersects $\Delta_0$ in the boundary (see \cite[Corollary 3.2]{ChenMoeller1}). 

Suppose that $\eta$ is trivial on $\PP^1_g(\mu)$. Then $\eta = e \delta_0$ on the partial compactification $\PP^1_g(\mu)\cup \Delta_0$ for some constant $e$. Since $\lambda = (\kappa_\mu / 12)\eta$ on $\PP^1_g(\mu)$, we have $\lambda = \ell \delta_0$ on $\PP^1_g(\mu)\cup \Delta_0$ for some constant $\ell$. Then the sum of Lyapunov exponents of a  
Teichm\"uller curve $C$ is given by $(\lambda \cdot C) / (\eta\cdot C) = \ell / e$ which is independent of $C$, contradicting that $\PP^1_g(\mu)$ is a varying stratum.  

Note that the above criterion is practically checkable for arithmetic Teichm\"uller curves generated by square-tiled surfaces by using the combinatorial description of area Siegel--Veech constants and sums of Lyapunov exponents (see~\cite[Section 2.5.1]{EKZ} and also~\cite[Theorem 1.8]{ChenSlope} for the relation with slopes of Teichm\"uller curves). 

The same conclusion would hold for a varying stratum of quadratic differentials, assuming that all Teichm\"uller curves contained in the stratum intersect only one irreducible boundary divisor (see \cite[Remark 4.7]{ChenMoeller2} for this assumption).  
\end{remark}

%%%%%%%%%%%%%%%%%
\begin{remark}
\label{remark:speculation}
%%%%%%%%%%%%%%%%%

We speculate that the rational Picard group (and the rational second cohomology group) of any varying stratum of holomorphic differentials $\PP^1_g(\mu)$ should be of rank one generated by $\eta$. 
\end{remark}

%%%%%%%%%%%%%%%%%
%%%%%%%%%%%%%%%%%
\section{Differentials of infinite area}
\label{sec:infinite}
%%%%%%%%%%%%%%%%%
%%%%%%%%%%%%%%%%%

In this section we prove Theorem~\ref{thm:infinite} for strata of $k$-differentials of infinite area. For completeness we also include the discussion for some strata in low genus.  

%%%%%%%%%%%%%%%%%
\subsection{Strata in genus zero}
%%%%%%%%%%%%%%%%%

Every stratum in genus zero is isomorphic to the corresponding moduli space of pointed smooth rational curves, which has a trivial Chow ring and is affine.  

%%%%%%%%%%%%%%%%%
\subsection{Strata in genus one}
\label{subsec:g1}
%%%%%%%%%%%%%%%%%

Since $\psi_i = 0$ on $\MM_{1,n}$ for all $i$ (see \cite[Theorem 2.2 c)]{ACPicard}), every stratum $\PP^k_1(\mu)$ in genus one (including the case of some $m_i = -k$ in $\mu$) has a trivial tautological ring and is affine by Proposition~\ref{prop:affine}. The affinity also follows from the fact that the ambient space $\MM_{1,n}$ is affine (see~\cite[Theorem 3.1]{ChenAffine}).  

%%%%%%%%%%%%%%%%%
\subsection{Strata in genus two}
\label{subsec:g2}
%%%%%%%%%%%%%%%%%

Since $\lambda$ is trivial on $\MM_2$, every stratum $\PP^k_2(\mu)$ in genus two with $m_i \neq -k$ for all $i$ has a trivial tautological ring and is affine by Proposition~\ref{prop:affine}.  

%%%%%%%%%%%%%%%%%
\subsection{Strata with poles of order $\leq - k$}
%%%%%%%%%%%%%%%%%

Suppose $m_1 \leq -k$. By the relations $\eta = (m_i + k)\psi_i$ for all $i$ and $(2g-2)\eta = k \kappa - \sum_{i=1}^n m_i \psi_i$, we obtain that 
\begin{eqnarray*}
k\Big(\kappa + \sum_{i=1}^n \psi_i\Big) & = & (2g-2+n) \eta \\
& = & (2g-2+n)(m_1+k) \psi_1. 
\end{eqnarray*}
It implies that $\kappa + \sum_{i=1}^n \psi_i + a \psi_1$ is trivial on $\PP^k_g(\mu)$ where $a = - (2g-2+n)(m_1+k) / k \geq 0$. Since $\kappa + \sum_{i=1}^n \psi_i + a \psi_1$ (as ample plus nef) remains ample on $\BM_{g,n}$, it follows that $\PP^k_g(\mu)$ is affine and Theorem~\ref{thm:infinite} is justified.  

%%%%%%%%%%%%%%%%%
%%%%%%%%%%%%%%%%%
\section{Extremal stratification of the Hodge bundle}
\label{sec:extremal}
%%%%%%%%%%%%%%%%%
%%%%%%%%%%%%%%%%%

In this section we prove Theorem~\ref{thm:extremal}. Recall the notation $\mu = (m_1, \ldots, m_n)$, $\mu' = (m_1+m_2, m_3, \ldots, m_n)$, and we work with unordered zeros.  The idea is to show that Teichm\"uller curves in $\PP_g^k\{\mu'\}$ has negative intersection numbers with the divisor class of $\BPP_g^k\{\mu'\}$ in $\BPP_g^k\{\mu\}$ for $k = 1, 2$. In order to verify that, we make use of the aforementioned fact that degenerate differentials in the boundary of a Teichm\"uller curve can have only simple polar nodes (after taking a local square root in the case of quadratic differentials). 

\begin{proof}[Proof of Theorem~\ref{thm:extremal}]
First consider the case of abelian differentials. The following relation of divisor classes holds on $\BPP^1_g\{\mu\}$: 
$$ 12 \lambda - D_h - \kappa_\mu \eta = (m_1+m_2+1) \Big(1 - \frac{1}{m_1+1} - \frac{1}{m_2+1} + \frac{1}{m_1+m_2+1}\Big) \BPP^1_g\{\mu'\}$$ 
modulo the other boundary divisors, where $D_h$ is the boundary divisor generically parameterizing stable differentials with simple polar nodes. This relation follows from~\cite[Proposition 6.3]{CCMKodaira}, where $\BPP^1_g\{\mu\}$ can be identified with the boundary divisor $D_\Gamma$ having $\Gamma_{\bot}$ as a rational vertex containing the two merged zeros and $\ell_\Gamma = m_1 + m_2 + 1$ therein. Note that the coefficient on the right of the above equality is positive for $m_1, m_2 \geq 1$. 

Let $C$ be the closure of a Teichm\"uller curve in $\PP^1_g\{\mu'\}$. Since $C$ does not intersect the part of the boundary of $\BPP^1_g\{\mu\}$ away from $D_h$ and $\BPP^1_g\{\mu'\}$, we obtain  
\begin{eqnarray*}
 \frac{(12 \lambda - D_h - \kappa_\mu \eta)\cdot C}{\eta \cdot C} & = & \kappa_{\mu'} - \kappa_{\mu} \\
 & = & -1 + \frac{1}{m_1+1} + \frac{1}{m_2+1} - \frac{1}{m_1+m_2+1}, 
\end{eqnarray*}
where $12\lambda - D_h = \kappa_{\mu'} \eta$ holds on the stratum $\BPP^1_g\{\mu'\}$ (modulo the other irrelevant boundary divisors of $\BPP^1_g\{\mu'\}$ which do not meet $C$). Consequently, this implies 
$$\frac{C\cdot \BPP^1_g\{\mu'\}}{C\cdot \eta} = - \frac{1}{m_1 + m_2 + 1} < 0. $$
Since $\eta$ has positive degree on every Teichm\"uller curve (e.g., from the positive sign of the area Siegel--Veech constant), it follows that 
$C\cdot \BPP^1_g\{\mu'\} < 0$ in $\BPP^1_g\{\mu\}$. Since Teichm\"uller curves are dense in $\BPP^1_g\{\mu'\}$ and the negative ratio above is independent of each individual Teichm\"uller curve, the desired claim follows by using the same argument as in \cite[Section 4]{Gheorghita}.  

The case of quadratic differentials is similar. The following relation of divisor classes holds on $\BPP^2_g\{\mu\}$: 
$$ 12 \lambda - D_h - \kappa_\mu \eta = (m_1+m_2+2)\Big(\frac{1}{2} - \frac{1}{m_1+2} - \frac{1}{m_2+2} + \frac{1}{m_1+m_2+2}\Big) \BPP^2_g\{\mu'\}$$ 
modulo the other boundary divisors, where the coefficient on the right is nonzero by the assumption that $m_1, m_2 \geq -1$ and at least one of them is positive. 
 Let $C$ be the closure of a Teichm\"uller curve in $\PP^2_g\{\mu'\}$. Since $C$ does not intersect the part of the boundary of $\BPP^2_g\{\mu\}$ away from $D_h$ and $\BPP^2_g\{\mu'\}$, we obtain  
\begin{eqnarray*}
 \frac{(12 \lambda - D_h - \kappa_\mu \eta)\cdot C}{\eta \cdot C}  
  = -\frac{1}{2} + \frac{1}{m_1+2} + \frac{1}{m_2+2} - \frac{1}{m_1+m_2+2}. 
\end{eqnarray*}
It follows that 
$$\frac{C\cdot \BPP^2_g\{\mu'\}}{C\cdot \eta} = - \frac{1}{m_1 + m_2 + 2} < 0. $$
The rest of the argument is the same as before. 
\end{proof}
  
 \begin{remark}
 \label{remark:extremal}
If a stratum in Theorem~\ref{thm:extremal} is disconnected, since each of its connected components contains a dense collection of Teichm\"uller curves, the conclusion of Theorem~\ref{thm:extremal} still holds for each connected component.  
\end{remark}
 
\begin{remark} 
For $k$-differentials with $k > 2$, since there is no meaningful way to define Teichm\"uller curves in this case, our method cannot be directly adapted. We leave it as an interesting question to determine whether the stratification of the $k$-th Hodge bundle is extremal for $k > 2$. 
 \end{remark} 
  
%%%%%%%%%%%%%%%%%%%

\end{document}